\documentclass{amsart}

\usepackage{amsfonts,amssymb,amsmath}
\usepackage{graphicx}
\usepackage{psfrag}
\usepackage{color,xcolor}
\usepackage{verbatim}
\usepackage{epstopdf}

\newtheorem{thm}{Theorem}[section]
\newtheorem{lem}[thm]{Lemma}

\newtheorem{cor}[thm]{Corollary}

\theoremstyle{definition}

\theoremstyle{remark}

\newtheorem{exam}{Example}[section]

\numberwithin{equation}{section}

\newcommand{\Z}{{\mathbb Z}}

\def\N{\mathcal{N}}

\def\R{\mathbb{R}}

\def\vep{\varepsilon}

\def\g{\mathcal{G}}

\def\bfc{\mathbf{c}}

\def\wdt{\widetilde}

\newcommand{\dist}{{\rm dist}}

\newcommand{\Var}{{\rm Var}}

\begin{document}

\title[Assouad dimension of the graph for Takagi function]{Assouad dimension of the graph for Takagi function}

\author{Lai Jiang}
\address{School of Fundamental Physics and Mathematical Sciences, Hangzhou Institute for Advanced Study, University of Chinese Academy of Sciences, Hangzhou 310024, China}
\email{jianglai@ucas.ac.cn}

\date{}


\begin{abstract}
For any integer $b\geq2$ and real series $\{c_n\}$ such that $\sum_{n=0}^\infty|c_n|<\infty$, the generalized Takagi function $f_{{\mathbf c},b}(x)$ is defined by
$$
	f_{{\mathbf c},b}(x):=\sum_{n=0}^\infty  c_n\phi(b^n x), \quad x\in [0,1],
$$
where $\phi(x)=\dist(x,\mathbb{Z})$ is the distance from $x$ to the nearest integer. 
The collection of functions with the form are called the Takagi class.
In this paper, we show that in the case that $\varlimsup_{n \to \infty} b^n |c_n|<\infty$, the Assouad dimension of the graph ${\mathcal G} f_{{\mathbf c},b}=\{(x,f_{{\mathbf c},b}(x)):x\in[0,1]\}$ for the generalized Takagi function $f_{{\mathbf c},b}(x)$ is equal to one, that is,
$$
\dim_A {\mathcal G} f_{{\mathbf c},b}=1.
$$
In particular, for each $0<a<1$ and integer $b \geq 2$, we define Takagi function $T_{a,b}$ as followed,
$$
	T_{a,b}(x):=\sum_{n=0}^\infty a^n \phi(b^n x),  \quad x\in [0,1].
$$
Then
$
	\dim_A {\mathcal G} T_{a,b}=1
$ 
if and only if $0<a \leq 1/b$.

\end{abstract}

\maketitle

\section{Introduction}
Takagi function, which is a nowhere differentiable function like Weierstrass function, has been studied extensively after being introduced by Takagi \cite{T1903}. In this paper, we focus on the Assouad dimension of the graph for Takagi function, and our main result gives the precise Assouad dimension.

\subsection{Takagi function}
It was a very well-known classical question whether continuous functions must be differentiable. Weierstrass \cite{W1872} constructed a famous nowhere differentiable function to give a negative answer for this question. Later, Takagi \cite{T1903} introduced another nowhere differentiable function defined by
$$T(x):=\sum_{n=0}^{\infty} \frac{\phi(2^n x)}{2^n}, \quad x \in [0,1],$$
where $\phi(x)=\dist(x,\mathbb{Z})$ is the distance from $x$ to the nearest integer.  Takagi \cite{T1903} proved its nowhere differentiability and Billingsley \cite{B82} gave a simplified proof later. 

The classical Takagi function $T(x)$ has attracted widespread attention. Hata and Yamaguti \cite{HY84} regarded the Takagi function as a solution of the discrete boundary value problem. Buczolich \cite{B08} found that the level set of the Takagi function is a finite set. Allaart and Kawamura \cite{AK11} studied further properties of these level sets.

There is a further generalization of the classical Takagi function, which expands its properties and applications. More precisely, for each integer $b\geq2$, the generalized Takagi function is defined by
$$
T_b(x):=\sum_{n=0}^\infty\frac{\phi(b^n x)}{b^n}, \quad x\in [0,1].
$$
When $b=2$, the function $T_2$ is the classical Takagi function. When $b=10$, the function $T_{10}$ is the van der Waerden function \cite{V30}. Baba \cite{B84} studied the maximum value of $T_b$. 
Shidfar and Sabetfakhri \cite{SS86} showed that $T_b$ is H$\ddot{\mbox{o}}$lder continuous with any order $\alpha<1$.
Allaart \cite{A14} studied the level sets of $T_b$.

{
Furthermore, let $a,b$ are real parameters such that $a<1$, $b>1$, $ab \geq 1$.
We can defined 
\begin{equation}
	T_{a,b}:=\sum_{n=0}^\infty a^n \phi(b^n x),  \quad x\in [0,1].
\end{equation}
}
Another direct generalization of the Takagi function is obtained by replacing
the factor $a^n$ with a sequence real constant $\{c_n\}_{n=0}^\infty$ such that $\sum_{n=0}^\infty |c_n|<\infty$. 
This gives functions of the form
\begin{equation}\label{eq:takagi-class}
	f_{\bfc,b}(x) := \sum_{n=0}^\infty c_n \phi(b^nx),\quad x\in [0,1] .	
\end{equation}
The collection of functions with the form in Eq.\eqref{eq:takagi-class} is called the \emph{Takagi class}.

Kôno \cite{K87} studied the continuity of $f_{\bfc,2}$.
If $\{2^n c_n\} \in \ell^2$, then $f_{\bfc,2}$ is absolutely continuous and hence differentiable almost everywhere.
If $\{2^n c_n\} \notin \ell^2$ and $\lim_{n \to \infty} 2^n c_n=0$, 
$f_{\bfc,2}$ is differentiable on an uncountably large set, while
$f_{\bfc,2}$ is not differentiable at almost
every point of $[0, 1]$.
If $\varlimsup_{n \to \infty }2^n|c_n| > 0$, then $f_{\bfc,2}$ is nowhere differentiable.

The signal Takagi function \cite{A13} is an important application of Takagi function, we give a example in end of this paper.


For each function $f$ defined on $D$, denote the graph of the function $f(x)$ by
$$
\g f:=\{(x,f(x)): x \in D\}.
$$
Note that for any integer $b \geq2$, the closed set $\g T_b\subset\mathbb{R}^2$ is a fractal set and both the Hausdorff dimension and box dimension of $\g T_b$ are equal to one, see, e.g., \cite{B15,KMY84}. However, the Assouad dimension of $\g f_{\bfc,b}$ and $\g T_{a,b}$ is still unknown and is computed for the first time in this paper.

\subsection{Assouad dimension}
We now recall the definition of the Assouad dimension. In our context, by writing $U(p,q,t)\lesssim V(p,q,t)$, we mean that there exists a constant $C>0$ which is independent on $p,q,t$ such that $U(p,q,t)\leq CV(p,q,t)$ for all $p,q,t$.

Let $d\geq1$ be a fixed integer used to represent dimensionality.
For any bounded set $E\subset\mathbb{R}^d$ and any $\delta>0$, a finite or countable collection of open sets $\{U_i\}_i$ is called a $\delta$-cover of $E$ if $E\subset\bigcup_{i} U_i$ and the diameter of each $U_i$ is not more than $\delta$: 
$$
	\mathrm{diam}(U_i)\leq\delta.
$$
Let $N_\delta(E)$ be the least number of the open sets in all possible $\delta$-covers of $E$. We denote the closed ball with center $x \in \R^d$ and radius $\rho>0$ by
$$B(x,\rho)=\{y \in \R^d : |y- x| \leq \rho\}.$$
Then for any bounded set $F\subset\mathbb{R}^d$, its Assouad dimension is defined by
$$
\dim_A F:=\inf\Big\{\alpha>0:\mbox{ for all }0<r<R\mbox{ and }x\in F,\,N_r\big(B(x,R)\cap F\big)\lesssim\Big(\frac{R}{r}\Big)^\alpha\Big\}.
$$
We refer the reader to \cite{F21} for more details of the Assouad dimension.

There is another equivalent definition of Assouad dimension by \cite{F90,F21}. For any $\delta>0$, a $\delta$-mesh or $\delta$-grid in $\R^d$ is the family of cubes of the form
$$
[m_1 \delta, (m_1+1) \delta ] \times [m_2 \delta,(m_2+1) \delta ]  \times \cdots \times [m_d \delta,(m_d+1) \delta ]	
$$
with integers $m_1,m_2, \ldots ,m_d\in\mathbb{Z}$. For any bounded set $E\subset\mathbb{R}^d$, let $\N_\delta(E)$ be the least number of the cubes in all possible $\delta$-meshs that cover $E$. We denote the closed cube with center $x=(x_1\ldots,x_d)\in\mathbb{R}^d$ and side length $2\rho$ by
$$Q(x,\rho)=[x_1-\rho,x_1+\rho]\times\cdots\times[x_d-\rho,x_d+\rho].$$
Then for any fixed positive integer $b\geq2$, we have
$$
\dim_A F=\inf\big\{\alpha>0:\mbox{ for all }n,m\in\Z^+\mbox{ and }x\in F,\,\N_{b^{-n-m}}\big(Q(x,b^{-n})\cap F\big)\lesssim b^{\alpha m}\big\}. 
$$
Note that here the value of Assouad dimension is independent of the choice of $b$.

For any bounded set $F\subset\mathbb{R}^d$, denote $\dim_H F$, $\dim_B F$, $\underline\dim_B F$ and $\overline\dim_B F$ the  Hausdorff dimension, box dimension, lower box dimension and upper box dimension of $F$ respectively. Note that 
\begin{align}\label{eq:dim_H<B<F}
	\dim_H F\leq\underline\dim_B F\leq\overline\dim_B F\leq\dim_A F.
\end{align}
See \cite{F90,F21} for this inequality as well as the definitions of Hausdorff dimension and box dimension. 
{The inequality in Eq.~\eqref{eq:dim_H<B<F} can be strict.}
Mitchell and Olsen \cite{MO18} constructed a fractal set $X$ by using iteration such that
$$
	\dim_H X <\underline\dim_B X<\overline\dim_B X<\dim_A X.
$$
  Yu \cite{Y20} proved that there exists Takagi function $T_{a,b}$ such that the box dimension is strictly smaller than the Assouad dimension for certain $a,b$. We refer the reader to \cite{ABK24, BC23, CFY22, R23, S18} for more details of the fractal dimensions.

\subsection{Main result}
We now turn to the graph $\g f_{\bfc,b}$ of the generalized Takagi function $f_{\bfc,b}$ with any integer $b \geq2$. Since for each integer $b\geq2$, the Hausdorff dimension and box dimension of the graph $\g T_b $ are equal to one \cite{B15,KMY84}, by Eq.~\eqref{eq:dim_H<B<F},
$$
\dim_A \g T_b \geq\dim_B\g T_b=\dim_H\g T_b=1.
$$
Our main result is the following.

\begin{thm}\label{thm:assouad}
For any integer $b\geq2$ and $\bfc=\{c_k\} $ such that 
$\varlimsup_{k \to \infty}  b^k|c_k |  <\infty$, we have
$$\dim_A\g T_{\bfc,b}=1.$$
\end{thm}

\begin{cor}\label{cor:assouad-old}
For each integer $b \geq2$, we have
$$\dim_A\g T_b=1.$$
\end{cor}

In the case that $ab>1$,
the box dimension of graphs of $T_{a,b}$ \cite{B15,KMY84} is equal to
$$
	\dim_B \g T_{a,b}=2+\frac{\log a}{\log b}>1.
$$

\begin{cor}\label{cor:assouad}
For each integer $b\geq2$, 
$\dim_A\g T_{a,b}=1$
if and only if $0 \leq a \leq b^{-1} $.
\end{cor}

\subsection*{Acknowledgements.} 
\medskip
The author would like to thank Prof. Huo-Jun Ruan and Prof. Yanqi Qiu for helpful discussions.

\section{The Assouad dimension of $T_{\bfc}$}
For the remainder of this paper, we fix integer $b \geq 2$.
Let $\bfc:=\{c_k\}_{k=0}^\infty$ be a sequence of real numbers such that
\begin{equation*}
	\varlimsup_{k \to \infty}  b^k|c_k |  <\infty.
\end{equation*}
Write $\eta=\max\big\{1,\varlimsup_{k \to \infty} b^k|c_k|\big\}$.

For any $n \in \Z^+$,
we define the partial sum sequences of $f_{\bfc,b}$ as
$$
	H_{n}(x):=\sum_{k=0}^{n-1} c_k \phi(b^k x), \quad x\in [0,1].
$$
For any $n,m \in \Z^+$,
we define another partial sum sequences of $f_{\bfc,b}$ as
$$
	H_{n,m}(x):=\sum_{k=n}^{n+m-1} c_k \phi(b^k x), \quad x\in [0,1].
$$
We denote by $S_{n}$ the set
$$
	S_{n}:=\big\{  (x,y):   x \in  [0,1] \mbox{ and } |H_{n}(x) -y | \leq \eta \cdot b^{-n}  \big\}.
$$
This section will explore the properties of these partial sums, which are essential for understanding the behavior of the Takagi function.

\begin{figure}[htbp]
\centering
\includegraphics[scale=0.4]{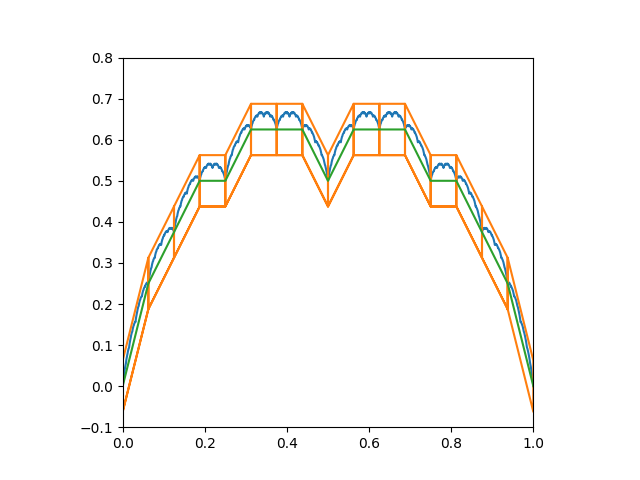}
\caption{Classical Takagi function $T$, $H_4$ and $S_{4}$.}
\end{figure}

\begin{lem}\label{lem:cover}
For any $n \in \mathbb{Z}^+$, we have 
$
	\g f_{\bfc,b} \subset S_{n}.
$

\end{lem}
\begin{proof}
Notice that $\phi(t) \leq 1/2$ for all $ t \in \R$.
Choose an arbitrary $x \in [0,1]$,
we have
$$ 
	\big|f_{\bfc,b}(x)-H_{n}(x)\big|
	=\bigg|\sum_{k=n}^\infty  c_k \phi(b^k x)\bigg|
	\leq \sum_{k=n}^\infty \frac{|c_k|}{2}
	\leq \sum_{k=n}^\infty \frac{\eta}{2b^k}
	= \frac{\eta \cdot b^{-n}}{2(1-\frac{1}{b})}
	 \leq \eta \cdot  b^{-n}. 
$$
Thus, $(x,f_{\bfc,b}(x)) \in S_{n}$.
For the arbitrariness of $x$, we have completed the proof.
\end{proof}

\begin{lem}\label{lem:linear}

For any $n \in \Z^+$ and $1 \leq i \leq 2r^n$, 
$H_n$ is  is linear on the interval $[\frac{i-1}{2 b^n} , \frac{i}{2 b^n} ]$.

\end{lem}

\begin{proof}

Let
$x_1=( i-1)/(2b^n)$ and $x_2=i/(2b^n)$.
Fix integer $0 \leq k \leq n-1$.
From
\[
  b^k x_1=\frac{i-1}{2b^{n-k}} \mbox{ and } b^k x_2=\frac{i}{2b^{n-k}},
\]
we observe that there is no point $x \in (x_1,x_2)$ such that $b^k x=j/(2b^{n-k})$ for some $j\in \Z$. 
Combining this with $n-k \geq 1$, we find that $0<\phi(b^k x) < 1/2$ for all $x\in (x_1,x_2)$. 
Summing over $k$ from $0$ to $n-1$,
$$
	H_n(x)=\sum_{k=0}^{n-1} c_k \phi(b^k x),
$$
is linear on the interval 
 $[x_1,x_2]$.  
Thus, the proof is complete.
\end{proof}

\begin{lem}\label{lem:lips}

For any $  n,m \in \Z^+$, function $H_{n}$ and $H_{n,m}$ are Lipschitz functions. More precisely, for any $x_1,x_2 \in [0,1]$, we have
$$
	\big|H_{n}(x_1)-H_{n}(x_2)\big|\leq n \eta |x_1-x_2|
	 \mbox{ and } 
	\big|H_{n,m}(x_1)-H_{n,m}(x_2)\big|\leq m \eta |x_1-x_2|.
$$
\end{lem}

\begin{proof}
For any $t_1,t_2 \in \R$, we have 
$$
	\big|\phi(t_1)-\phi(t_2)\big| =\big|\dist(t_1,\Z)-\dist(t_2,\Z)\big| \leq \big|\dist(t_1,t_2)\big|=|t_1-t_2|.
$$
Hence, for any $k \in \Z^+$ and $x_1,x_2 \in [0,1]$, we have
\begin{align*}
	\big| c_k\phi ( b^k x_1 )-c_k\phi ( b^k x_2)    \big| 
	 \leq \big| c_k b^k (x_1 - x_2 )    \big|
	 \leq  \eta |x_1-x_2|   .
\end{align*}
Summing over $k$ from $0$ to $n-1$,
$$
	\big| H_n(x_1)-H_n(x_2)\big|
	=\bigg|\sum_{k=0}^{n-1} c_k \phi(b^k x)- c_k \phi(b^k x)\bigg|
	\leq \sum_{k=0}^{n-1} \big| c_k b^k (x_1 - x_2 )    \big|
	\leq n \eta |x_1 - x_2 |.
$$
Similarly, summing over $k$ from $n$ to $n+m-1$,
$$
	\big| H_{n,m}(x_1)-H_{n,m}(x_2)\big|
	\leq \sum_{k=n}^{n+m-1} \big| c_k b^k (x_1 - x_2 )    \big|
	\leq m \eta |x_1 - x_2 |.
$$
\end{proof}

Let $O(g,E)=\sup_{x,x' \in E } \big|g(x)-g(x')\big|$ be the \emph{oscillation} of the function $g$ on set $E$.
Let $O(g,\emptyset)=0$ by default.

We can quickly make connection between $\N_r(\g g)$ and the oscillation of $g$, which is widely used in obtaining the box dimension of the graphs of continuous functions.
Similar results can be find in \cite{B15,F90}.
\begin{lem}\label{lem:osci}
Let $d \in \R$ and $r >0 $.
Assume $g$ is a continuous function defined on $[d,d+r]$, then we have
$$
	\N_r(\g g ) \leq O\big(g,[d,d+r]\big)/r+2.
$$
\end{lem}

  \begin{lem}\label{lem:tkey-new}

For any $n \in \mathbb{Z}^+\cup \{0 \}$, $m \in \Z^+$, $1 \leq i \leq b^n$, and $y \in \R$, we have
$$
	\N_{b^{-n-m}} \bigg(S_{n+m} \cap \Big(\Big[\frac{i-1}{b^n}, \frac{i}{b^n}\Big] \times \Big[y-\frac{\eta}{b^{n}},y+\frac{\eta}{b^{n}}\Big] \Big) \bigg)  \leq  (10\eta +m \eta +4  )  b^m.
$$
\end{lem}
\begin{proof}

Fix $n \in \mathbb{Z}^+\cup \{0 \}$, $m \in \Z^+$, and $1 \leq i \leq b^n$.
For each $1 \leq j\leq b^m$, we write 
$$
	I_{j}=\Big[\frac{i-1}{b^n}+\frac{j-1}{b^{n+m}}    ,\frac{i-1}{b^n}+\frac{j}{b^{n+m}}   \Big].
$$
Define
$$
 	R_j=I_{j} \times [y  -  \eta  \cdot  b^{-n},y+ \eta  \cdot  b^{-n}].
$$
By the definition of $S_{n+m}$, we can see that if 
$$
  \g H_{n+m}\cap R_j \subset I_{j}\times \Big[ \frac{p}{b^{n+m}}, \frac{q}{b^{n+m}}\Big]
$$
for some $p,q\in \R$, then 
$$
  S_{n+m}\cap R_j \subset I_{j}    \times \Big[ \frac{p-\eta}{b^{n+m}}, \frac{q+\eta}{b^{n+m}}\Big].
$$
Hence,  
$$
	\N_{b^{-n-m}}(S_{n+m}\cap R_j) \leq 2+2\eta+ \N_{b^{-n-m}}(\g H_{n+m} \cap R_j).
$$
Let $I=[\frac{i-1}{b^n}    ,\frac{i-1}{b^n}] $ and $\wdt{R}=I \times [y-b^{-n},y+b^{-n}]$.
By summing $j$ over $1$ to $b^m$,
\begin{equation}\label{eq:Lem4.7-1}
  \N_{b^{-n-m}}(S_{n+m}\cap \wdt{R}) \leq (2+2\eta)b^m+ \sum_{j=1}^{b^m} \N_{b^{-n-m}}(\g H_{n+m} \cap R_j).
\end{equation}

Fix $y \in \R$.
We write 
$$D:=D(y,n)=\big\{x\in I:  |H_{n}(x)-y| \leq 2\eta\cdot b^{-n}  \big\}.$$
Let $J_1=[\frac{i-1}{b^n}    ,\frac{2i-1}{2b^n}] $ and $J_2=[\frac{2i-1}{2b^n}    ,\frac{i}{b^n}] $. 
It is clear that 
$$I=J_1 \cup J_2 =\bigcup_{j=1}^{b^m}  I_j.    $$
From Lemma~\ref{lem:linear}, $H_{n}$ is linear on both $J_1$ and $J_2$.
Thus, we have
\begin{align}
	&\sum_{j=1}^{b^m} O\big(H_{n}, I_{j} \cap D \big) \label{eq:count-1}  \\
	\leq&\Var\big(H_{n},J_1\cap D\big)+\Var\big(H_{n}, J_2 \cap D\big)\\
	\leq & 4\eta\cdot b^{-n}+ 4\eta\cdot b^{-n} = 8\eta\cdot b^{-n},
\end{align}
where $\Var(g,E)$ represents the variation of $g$ on $E$.
From Lemma~\ref{lem:lips}, $H_{n,m}$ is Lipschitz on $I$.
Moreover, we have
\begin{equation}\label{eq:count-2}
	\sum_{j=1}^{b^m} O\big(H_{n,m}, I_{j}  \big)
	\leq \sum_{j=1}^{b^m} m  \eta   \cdot  | I_{j} |
	=m\eta  \cdot  |I|=  m \eta \cdot b^{-n}.
\end{equation}


From Lemma~\ref{lem:cover}, for any $ x \notin D$, we have 
$$
	\big| H_{n+m}(x) -y\big|\geq \big| H_{n}(x)-y \big|-\big| H_{n}-H_{n-m}(x) \big| > 2\eta \cdot b^{-n}-\eta \cdot b^{-n}=\eta \cdot b^{-n},
$$
which implies that
$	
	H_{n+m}(x) \notin [y-\eta \cdot b^{-n},y+ \eta \cdot b^{-n}].
$
Hence, we have
$$
	\g H_{n+m} \cap R_j \subset   \g H_{n+m} \cap \big( ( I_j\cap D)\times \R\big)   
$$
Combining with Lemma~\ref{lem:osci}, we have
\begin{align*}
	& \N_{b^{-n-m}}(\g H_{n+m} \cap R_j)  \\
	\leq & O( H_{n+m},I_j\cap D)/b^{-n-m}+2   \\ 
	\leq &b^{n+m} \big( O( H_{n},I_j\cap D)+O( H_{n,m},I_j\cap D)\big)+2  .
\end{align*}
Combining with Eq.~\eqref{eq:count-1} and Eq.~\eqref{eq:count-2}, by
summing $j$ over $1$ to $b^m$, we have
\begin{align*}
 &\sum_{j=1}^{b^m} \N_{b^{-n-m}}(S_{n+m} \cap R_j) \\
  \leq & (2\eta+2)  b^m+2b^m+ \sum_{j=1}^{b^m} b^{n+m}\Big(  O (H_{n}, I_{j} \cap D ) + O(H_{n,m}, I_j \cap D ) \Big) \\
  \leq & (2\eta+4) b^m+  b^{n+m}(  8\eta \cdot b^{-n}+m\eta \cdot b^{-n})\\
  =&(10\eta +m \eta +4  )b^m.
\end{align*}
This completes the proof of the lemma.
\end{proof}

\begin{proof}[Proof of Theorem~\ref{thm:assouad}]

For any $x_0 \in [0,1]$ and $n \in \mathbb{Z}^+$,
there exists $0 \leq i \leq b^n$ such that
$$\Big[x_0-\frac{1}{b^n},x_0+\frac{1}{b^n}\Big]  \subset \Big[ \frac{i-1}{b^n},\frac{i+2}{ b^n} \Big]. $$
Let $y_0=f_\bfc(x_0)$, then we have
$$
	Q\big((x_0,y_0),b^{-n}\big) 
	\subset \Big[ \frac{i-1}{b^n},\frac{i+2}{ b^n} \Big] \times  \Big[y_0- \frac{\eta}{b^n},y_0+\frac{\eta}{ b^n} \Big].
$$

From Lemma~\ref{lem:tkey-new}, for any $m \in \mathbb{Z}^+$, we have
\begin{align*}
	&\N_{b^{-n-m}} \Big( \g f_{\bfc,b} \cap Q\big((x_0,y_0),b^{-n}\big) \Big) \\
	\leq& \sum_{\ell=i}^{i+2} \N_{b^{-n-m}} \bigg( S_{n+m} \cap \Big(\Big[\frac{\ell-1}{b^n},\frac{\ell}{b^n}\Big] \times \Big[y_0- \frac{\eta}{b^n},y_0+\frac{\eta}{b^n} \Big] \Big) \bigg) \\
	\leq &  3(10\eta +m \eta +4  )  b^m.
\end{align*}

For any $\vep>0$, there exists a constant that $C_\vep>0$ such that 
$$C_\vep b^{m\vep} \geq 3(10\eta +m \eta +4  )  ,\quad \forall m \in \Z^+.$$ 
Thus,
$$\N_{b^{-n-m}} \Big( \g f_{\bfc,b} \cap Q\big( x,b^{-n}\big) \Big) \leq C_\vep b^{(1+\vep)m},$$
for all $x \in \g T_{a,b}$ and $n,m \in \Z^+$.
This implies $1+\vep$ lies in the following set:
$$
	\big\{ \alpha:  \mbox{ for all } n,m \in \Z^+ \mbox{ and } x \in \g f_{\bfc,b} ,
\N_{b^{-n-m}}\big(Q(x,b^{-n})\cap  \g f_{\bfc,b}  \big) \lesssim  b^{\alpha m}  \big\}.
$$
Therefore, we have $ \dim_A \g  f_{\bfc,b} \leq  1 +\vep$.
For the arbitrariness of $\vep$, it follows that
$$\dim_A \g f_{\bfc,b}  \leq 1.$$
On the other hand,
it is clear that $\dim_A \g  f_{\bfc,b} \geq \underline\dim_B \g  f_{\bfc,b} \geq 1$.
Thus, $$\dim_A \g  f_{\bfc,b}=1.$$
\end{proof}

\begin{exam}
The signal Takagi function\cite{A13}, with the following form
$$
	f_{\mathbf r}(x):=\sum_{n=0}^\infty \frac{r_n}{2^n} \phi(2^n x), \quad x \in [0,1],
$$
where $r_n =\pm 1$ for each $n$.
We have $\eta=1$.
The Assouad dimension of graph of $f_{\mathbf r}$ is one, that is, 
$$ \dim_A \g f_{\mathbf r}=1.$$

\begin{figure}[htbp]
\centering
\includegraphics[scale=0.4]{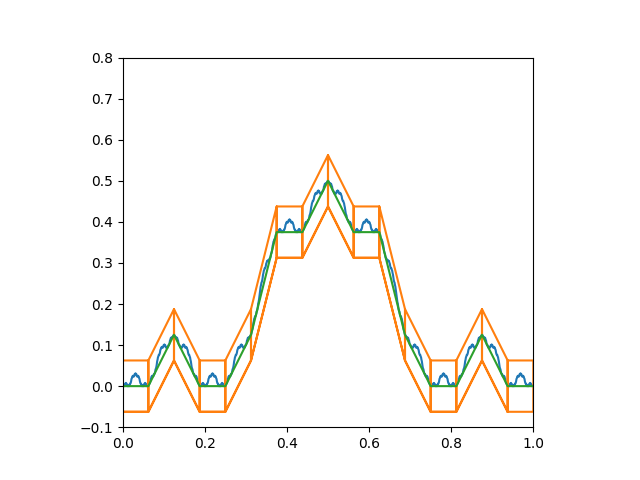}
\caption{Signal Takagi function $f_{\mathbf r}$, $H_4$ and $S_{4}$, where $r_n=(-1)^n$.}
\end{figure}

\end{exam}





\bibliographystyle{amsplain}

\end{document}